\documentclass[graybox]{svmult}

\usepackage{helvet}         % selects Helvetica as sans-serif font
\usepackage{courier}        % selects Courier as typewriter font
\usepackage{type1cm}        % activate if the above 3 fonts are
\usepackage{makeidx}         % allows index generation
\usepackage{graphicx}        % standard LaTeX graphics tool
\usepackage{multicol}        % used for the two-column index
\usepackage[bottom]{footmisc}% places footnotes at page bottom

\usepackage{amsmath}
\usepackage{amssymb}
\usepackage{mathrsfs}
\usepackage{amsfonts}\usepackage{amsxtra}\usepackage{color}

\makeindex             % used for the subject index
\DeclareMathOperator{\G}{\mathcal{G}}

\DeclareMathOperator{\Hom}{Hom}

\DeclareMathOperator{\FFP}{FFP}

\DeclareMathOperator{\Null}{null}

\DeclareMathOperator{\trace}{trace}

\DeclareMathOperator*{\spann}{span}\DeclareMathOperator{\tr}{\trace}

\DeclareMathOperator{\Pol}{Pol}

\DeclareMathOperator{\Tr}{Tr}
\DeclareMathOperator{\sym}{sym}

\newcommand{\K}{\mathcal{K}}

\newcommand{\R}{\mathbb{R}}\newcommand{\N}{\mathbb{N}}

\begin{document}

\title*{Cubatures on Grassmannians: moments, dimension reduction, and related topics}
 \titlerunning{Cubatures on Grassmannians}
% Use \titlerunning{Short Title} for an abbreviated version of
% your contribution title if the original one is too long
\author{Anna Breger \and Martin Ehler \and Manuel Gr\"af \and Thomas Peter}
\authorrunning{Cubatures in Grassmannians} 
% your contribution title if the original one is too long
\institute{A.~Breger \at University of Vienna, Department of Mathematics, Oskar-Morgenstern-Platz 1, A-1090 Vienna, 
\email{anna.breger@univie.ac.at}
\and M.~Ehler \at University of Vienna, Department of Mathematics, Oskar-Morgenstern-Platz 1, A-1090 Vienna, 
\email{matin.ehler@univie.ac.at}
\and M.~Gr\"af \at Acoustics Research Institute, Wohllebengasse 12-14, A-1040 Vienna, 
\email{mgraef@kfs.oeaw.ac.at}
\and T.~Peter \at University of Vienna, Department of Mathematics, Oskar-Morgenstern-Platz 1, A-1090 Vienna, 
\email{petert@uni-osnabrueck.de}
}
%
% Use the package "url.sty" to avoid
% problems with special characters
% used in your e-mail or web address
%
\maketitle

%\abstract{...}

\section{Introduction}\label{sec:1}
Function approximation, integration, and inverse problems are just few examples of numerical  fields that rely on efficient strategies for function sampling. As particular sampling rules, the concepts of cubatures in Euclidean space and the sphere have been widely investigated to integrate polynomials by a finite sum of sampling values, cf.~\cite{Delsarte:1977aa,Engels:1980uq,Hoggar:1982fk,Konig:1999fk,Neumaier:1988kl}. To some extent, cubatures are universal sampling strategies in the sense that they are highly efficient in many fields, and in the context of function approximation, covering, and integration they have proved superior to the widely used random sampling \cite{Brauchart:fk,Reznikov:2015zr}. 
%To some extent, cubatures are universal sampling strategies in the sense that for various different quantities, such as integration, function approximation, and covering, they are highly efficient or at least more efficient than random sampling, \cite{Brauchart:fk,Reznikov:2015zr}.
% see, for instance, \cite{Reznikov:2015zr}\cite{Brauchart:fk} and references therein. 

Recently, cubatures on compact manifolds have attracted attention, cf.~\cite{Brandolini:2014oz,Filbir:2010aa,Pesenson:2012fp}. Integration, covering, and polynomial approximation from cubatures on manifolds and homogeneous spaces have been extensively studied from a theoretical point of view, cf.~\cite{deBoor:1993aa,DeVore:1993ab,Geller:2011fk,Maggioni:2008fk,Reznikov:2015zr} and references therein. Orthogonality is a leading concept in many mathematical fields, and dimension reduction is intrinsically tied together with low dimensional projections. The Grassmannian manifold is the space of orthogonal projectors of fixed rank, and in this chapter we like to explore on the concept of cubatures in Grassmannians. Therefore, we shall provide a brief overview of  recent results on Grassmannian cubatures.  %with a special focus on moment reconstruction problems. 

Our starting point in Section \ref{sec:2} is the problem of reconstructing a sparse (i.e.~finitely supported) probability measure $\mu$ in $\R^d$ from its first few moments.  Sparse distributions are indeed uniquely determined by their first few moments, and Prony's method has recently been adapted to this reconstruction, \cite{P95,KPRO16}. According to the Johnson Lindenstrauss Lemma, low dimensional projections of $\mu$ still capture essential information, \cite{Dasgupta:2003fk}. 
%There is also ongoing research towards reconstruction of sparse probability measures from the first few low-dimensional methods, i.e., taking a couple of low dimensional projectors and then the moments from those. 
Taking the first few moments of low dimensional projections only, we now aim to reconstruct the first few moments of $\mu$ but we allow for general probability distributions in Section \ref{sec:3}, not necessarily sparse ones, cf.~\cite{Graf:2014qd}. A new construction of suitable projections are provided in Theorem \ref{th:con Q}. It turns out that the choice of projectors is closely linked to cubatures in Grassmannians, i.e., the set of low dimensional projectors should form a cubature, see Section \ref{sec:4} and, in particular,  Theorem \ref{th:zwei}. Hence, the reconstruction of high dimensional moments from lower dimensional ones is naturally related to the concept of Grassmannian cubatures. Therefore, we then discuss in Section \ref{subsec:5.1} numerical constructions of cubatures in Grassmannians by minimizing the worst case integration error of polynomials, cf.~\cite{Breger:2016vn,Bachoc:2010aa}. In Section \ref{sec:5.2}, we go beyond polynomials and briefly discuss sequences of low cardinality cubatures that yield optimal worst case integration error rates for Bessel potential functions, cf.~\cite{Brandolini:2014oz}, see also \cite{Breger:2016vn}. The optimal integration errors of cubatures directly induce schemes for function approximation from samples by replacing  the inner products of $L_2$ orthonormal basis expansions with cubature approximations, see Section \ref{sec:5.3}. Intuitively, good samplings for function approximation should well cover the underlying space, and, in Section \ref{Sec:5.4}, we recapitulate that sequences of low cardinality Grassmannian cubatures are asymptotically optimal coverings, cf.~\cite{Breger:2016rc}. To further reflect on the versatility, we also provide some results on phase retrieval problems, in which Grassmannian cubatures are used, see Section \ref{sec:5.5} for details. 

So far, we have outlined the use of Grassmannian cubatures for various topics in numerical mathematics. 
Unions of Grassmannians would offer additional flexibility since the rank of the projector is not needed to be fixed a-priori. Projectors with varying ranks may indeed have benefits in practice, see \cite{Harandi:2013wo,Veeraraghavan:2008aa} for potential applications. Therefore, the concept of cubatures on unions of Grassmannians is discussed in Section \ref{sec:6}. The number of required cubature points is mainly steered by the dimension of the underlying polynomial space. Addressing necessary prerequisites for the aforementioned topics within unions of Grassmannians, i.e., approximation of integrals and functions, moment reconstruction, covering, and phase retrieval, we shall determine the dimensions of polynomial spaces on unions of Grassmannians, cf.~\cite{Ehler:2014zl}. For special cases, we provide elementary proofs. The general cases need deeper analysis, for which we refer to \cite{Ehler:2014zl}.

\section{Reconstruction from moments and dimension reduction}\label{sec:2}
\subsection{Reconstructing sparse distributions from moments}
Our starting point is a high dimensional random vector $X\in\R^d$ with finite support $\{x_i\}_{i=1}^m\subset \R^d$, i.e., $X$ is distributed according to a discrete probability measure $\mu$ on $\R^d$ with support $\{x_i\}_{i=1}^m$ and positive weights $\{a_i\}_{i=1}^m$ satisfying $\sum_{i=1}^m a_i = 1$, so that
\begin{equation*}
\mu = \sum_{i=1}^m a_i \delta_{x_i},
\end{equation*}
where $\delta_{x_i}$ denotes the point measure at $x_i$. 
%Suppose $X\in\R^d$ is a random vector distributed according to $\mu$, denoted by $X\sim\mu$. 
We now aim to reconstruct $\mu$ from knowledge of the moments 
\begin{equation}\label{eq:mom def}
m_\mu(\lambda):=\mathbb{E} X^\lambda = \sum_{i=1}^m a_i x_i^\lambda ,\quad \lambda\in\Lambda, 
\end{equation}
where $\Lambda \subset \N^d$ is some fixed subset. The nonlinear inverse problem of reconstructing $\mu$ means to identify its support $\{x_i\}_{i=1}^m$ and its weights $\{a_i\}_{i=1}^m$. The core idea of Prony's method is to determine an ideal $\mathcal{I}$ of polynomials on $\R^d$ just from the moments  $m_\mu(\lambda)$, $\lambda\in\Lambda$, through a system of linear equations, such that its zero locus 
\begin{equation*}
\mathcal{V}(\mathcal{I}) = \{x\in\R^d : f(x)=0, \;\forall f\in\mathcal{I}\}
\end{equation*}
is exactly the point set $\{x_i\}_{i=1}^m$. The one dimensional case, expressed in terms of difference equations, was introduced in \cite{P95}, the multivariate case is treated in \cite{KPRO16}.

Once $\mathcal{I}$ is determined, its zero locus $\mathcal{V}(\mathcal{I})=\{x_i\}_{i=1}^m$ can be determined by standard methods \cite{B13}, and the weights $\{a_i\}_{i=1}^m$ are computed by a system of linear equations from the Vandermonde system \eqref{eq:mom def}. 

More specifically, the zero locus $\mathcal{V}(\mathcal{I}_i)$ of each ideal
\begin{equation*}
\mathcal{I}_i := \big( (z-x_i)^\alpha : \alpha\in\N^d,\; |\alpha|=1 \big) 
\end{equation*}
is $\mathcal{V}(\mathcal{I}_i)=\{x_i\}$, for $i=1,\ldots,m$, so that 
$
\{x_i\}_{i=1}^m = \mathcal{V}(\mathcal{I})$ with $\mathcal{I} := \mathcal{I}_1\cdots \mathcal{I}_m$. 
Note that $\mathcal{I}$ coincides with 
\begin{equation*}
\mathcal{I} = \Big (\prod_{i=1}^m (z-x_i)^{\alpha_i}: \alpha_i \in\N^d,\;|\alpha_i|=1,\; i=1,\ldots,m\Big),
\end{equation*}
so that we have $d^m$ many generators of the ideal that must now be determined from the moments $m_\mu(\lambda)$, $\lambda\in\Lambda$. 

To simplify, let us now suppose that $d=1$. In this case, the ideal $\mathcal{I}$ is generated by the single polynomial 
\begin{equation*}
p(z)=(z-x_1)\cdots (z-x_m)=\sum_{k=0}^{m} p_k z^k
\end{equation*}
of degree $m$. Its coefficient sequence $\{p_k\}_{k=0}^m$ satisfies 
\begin{equation}\label{eq:prony}
\sum_{k=0}^m p_k m_\mu(k+\lambda)  =  \sum_{i=1}^m x_i^\lambda a_i \sum_{k=0}^m p_k  x_i^k  =  \sum_{i=1}^m x_i^\lambda a_i p(x_i) =0.
\end{equation}
Equation \eqref{eq:prony} holds for arbitrary values of $\lambda$. Thus, varying $\lambda$ und using that $p_m = 1$ leads to the linear system of equations
\begin{equation}\label{eq:solv}
\sum_{k=0}^{m-1} p_k m_\mu(k+\lambda) = -m_\mu(m+\lambda), \quad \lambda\in\Lambda',
\end{equation}
where $\Lambda'\subset\Lambda$ such that $k+\Lambda'\in\Lambda$, for all $k=0,\ldots,m$. We now attempt to solve \eqref{eq:solv} for $p_0,\ldots,p_{m-1}$. Obviously, $\Lambda$ must be sufficiently large, so that 
\begin{equation}\label{eq:rank matrix}
H := \big( m_\mu(k+\lambda)  \big)_{\substack{\lambda\in\Lambda' \\ k=0,\ldots,m-1 }} \in\R^{|\Lambda'| \times {m}}
\end{equation}
can have full rank $m$. From knowledge of $p$, the eigenvalues of its companion matrix yield its zeros $\{x_i\}_{i=1}^m$. Having determined $\{x_i\}_{i=1}^m$, \eqref{eq:mom def} yields a Vandermonde system of linear equations to compute the weights $\{a_i\}_{i=1}^m$. Note that the rank condition in \eqref{eq:rank matrix} is satisfied for $\Lambda=\{0,\ldots,2m-1\}$ and $\Lambda'=\{0,\ldots,m-1\}$, cf.~\cite{PP13} and \cite{PT14} for an overview of Prony methods.

The case $d>1$ is more involved, but can essentially be treated similarly. In \cite{KPRO16} it is shown that $\#\Lambda = \mathcal O(m^d)$ suffices to ensure reconstruction while $\#\Lambda = \mathcal O(md)$ suffices if $\{x_i\}_{i=1}^m$ are in general position.
%We refer to \cite{} for the relations between $m$, $d$, and $\Lambda$ to ensure reconstruction. 

%A few words on stability are in order since the Prony method is falsely infamous for being unstable. 
Concerning numerical stability, 
%At first 
one has to differentiate between the idea of Prony's method as presented here and stable numerical variants for implementation as for example ESPRIT \cite{RK86}, MUSIC \cite{S86}, and finite rate of innovation \cite{VMB02}. These algorithms perform excellent in applications. 
%We now want to give a short impression of the cases concerning $\mu$ that are numerically challenging.
If $\lambda$ and $k$ are chosen as proposed in \eqref{eq:prony}, the system matrix \eqref{eq:rank matrix}
is a Hankel matrix that can be factored into 
\begin{equation*}
H = A^\top DA
\end{equation*}
 with a diagonal matrix $D = \mathrm{diag}(a_i)_{i = 1}^m$ and a Vandermonde matrix $A = (x_i^k)_{k=0,i=1}^{m-1,m}$. For $d>1$, a similar factorization holds, where $A$ is a generalized Vandermonde matrix.  
% A large condition number of $A$ yields 
% With this factorization we observe that  $\mathrm{cond}_2(H)\gg 1$ if $\mathrm{cond}_2(A)\gg 1$. 
% 
Due to the Vandermonde structure of $A$, its condition number tends to be large if the minimal separation distance $\nu_\mu$ is small or if there are large sphere deviations $\alpha_\mu$, i.e., 
 \begin{equation*}
 \nu_\mu:=\min_{i\neq j}\|x_i-x_j\|_2 \approx 0,\quad \text{or}\quad \alpha_\mu:=\max_{i\neq j}\big|\|x_i\|-\|x_j\|\big|\gg 0.
 \end{equation*}
This pinpoints stable performances when the measure $\mu$ has a well-separated support on a sphere with well behaved weights.

Note that the Prony method works beyond probability measures and can deal with $x_i \in \mathbb C^d$, $a_i \in \mathbb C$ and to this end also with $\lambda \in \mathbb Z^d$. Indeed, if $\Lambda'$ in \eqref{eq:prony} is chosen as $\Lambda'\subset -\mathbb N^d$, then the resulting system matrix becomes a Toeplitz matrix, which is preferred in some literature on Prony's method.

\subsection{Dimension reduction}
%Given a random vector $X\in\R^d$, we now look at its lower dimensional projection $PX$, where $P$ is an orthogonal projector of rank $k$, i.e., $P$ is an element in the Grassmannian space
%\begin{equation*}
%\G_{k,d} := \{ P \in \R^{d\times d}_{\sym} \;:\; P^{2}=P ;\; \Tr(P)=k \},
%\end{equation*}
%where $\R^{d\times d}_{\sym}$ is the set of symmetric matrices in $\R^{d \times d}$. 
%
The idea of dimension reduction is that properties of interest of the high dimensional random vector $X\in\R^d$ may still be captured within its orthogonal $k<d$ dimensional projection, i.e., in $PX$, where $P$ is an element in the Grassmannian space
\begin{equation*}
\G_{k,d} := \{ P \in \R^{d\times d}_{\sym} \;:\; P^{2}=P ;\; \Tr(P)=k \}.
\end{equation*}
Here $\R^{d\times d}_{\sym}$ is the set of symmetric matrices in $\R^{d \times d}$.   
Consider now two sparsely distributed random vectors  
\begin{equation}\label{eq:XY distr}
X,Y\sim \sum_{i=1}^m a_i \delta_{x_i}. 
\end{equation}
Their difference $X-Y$ is distributed according to 
\begin{equation*}
X-Y \sim \sum_{i,j=1}^m a_ia_j \delta_{x_i-x_j}.
\end{equation*}
For $P\in\G_{k,d}$, the magnitude of the differences is distributed according to 
\begin{equation*}
\|PX-PY\|^2 \sim \sum_{i,j=1}^m a_ia_j \delta_{\|Px_i-Px_j\|^2}.
\end{equation*}
In fact, for $0<\epsilon<1$ and $k$ with $d\geq k\geq \frac{4\log(m)}{\epsilon^2/2-\epsilon^3/2}$, 
%and given $X,Y$ as in \eqref{eq:XY distr}, 
there is $P\in\G_{k,d}$, such that 
\begin{equation}\label{eq:JLL}
(1-\epsilon)\|X-Y\|^2\leq \frac{d}{k}\| P X - PY \|^2 \leq (1+\epsilon)\|X - Y\|^2
\end{equation}
holds with probability $1$. This is the direct consequence of realizations of the Johnson-Lindenstrauss Lemma applied to the deterministic point set $\{x_i\}_{i=1}^m$, cf.~\cite{Dasgupta:2003fk}. 

Note that \eqref{eq:JLL} tells us that the dimension reduction still preserves essential information of $X$ and $Y$.  At this point though, we just know of its existence, and we have not yet specified any particular projector $P$ such that \eqref{eq:JLL} holds, see \cite{Dasgupta:2003fk,Achlioptas:2003wo,Matousek:2008al} for different types of random choices. 

We should point out that $PX$ and $PY$ are contained in a $k$ dimensional subspace of $\R^d$, but still have $d$ entries as vectors in $d$ dimensions. The actual dimension reduction takes place by applying $Q\in \mathcal{V}_{k,d}$ with $Q^\top Q = P$, where 
\begin{equation*}
\mathcal{V}_{k,d}:=\{ Q\in\R^{k\times d}_{\sym} : QQ^\top = I_k\}
\end{equation*}
denotes the Stiefel manifold. The inequality \eqref{eq:JLL} becomes
\begin{equation*}
(1-\epsilon)\|X-Y\|^2\leq \frac{d}{k}\| Q X - Q Y \|^2 \leq (1+\epsilon)\|X - Y\|^2,
\end{equation*}
where $QX, QY\in\R^k$ are properly dimension reduced random vectors still containing the information of the pairwise differences up to a factor $1\pm \epsilon$. 

\section{High dimensional moments from lower dimensional ones}\label{sec:3}
We shall now combine dimension reduction with a modified problem, which is related to the reconstruction from moments. First, we drop the sparsity conditions and allow for arbitrary probability measures $\mu$ on $\R^d$.  Let $X\in \R^d$ be some random vector with unknown Borel probability distribution on $\R^d$. Suppose we do not have access to its moments, but we observe the first few moments of order $T$ of low-dimensional linear projections, i.e., for $\{Q_j\}_{j=1}^n\subset \mathcal{V}_{k,d}$, we measure 
\begin{equation}\label{eq:2}
\mathbb{E} (Q_jX)^s , \quad s\in\N^k,\; |s|\leq T,
\end{equation}
We cannot reconstruct $\mu$ directly, but we aim to determine the first few high-dimensional moments
\begin{equation}\label{eq:1}
\mathbb{E}X^r, \quad r\in\N^d,\; |r|\leq T. 
\end{equation}
In other words, we know the first few moments of order $T$ of the dimension reduced random vectors $Q_jX\in\R^{k}$, $j=1,\ldots,n$ and our task is to reconstruct the high-dimensional moments, cf.~\cite{Graf:2014qd}. 
The idea is to interpret %the corresponding quantities 
moments as algebraic polynomials and represent desired high degree polynomials as products of polynomials of lower degree. The remainder of this chapter is dedicated to establish precise relations. 

Polynomials of total degree $T$ on $\R^d$, denoted by $\Pol_T(\R^d)$, are decomposed by 
\begin{equation*}
\Pol_T(\R^d)= \bigoplus_{t=0}^T \Hom_t(\R^d),
\end{equation*}
where $\Hom_t(\R^d)$ denotes the space of homogeneous polynomials of degree $t$ on $\R^d$. 
Let $x\in\R^d$ be a vector of unknowns, then $(Q_jx)^s$ is a homogenous polynomial of degree $|s|$. If 
\begin{equation}\label{eq:generater}
\{(Q_jx)^{s} : j=1,\ldots,n,\; s\in\N^{k},\; |s|= t\}
\end{equation}
spans $\Hom_t(\R^d)$, then each monomial of order $t$ is a linear combination of elements in \eqref{eq:generater}, so that the linearity of the expectation yields that all high-dimensional moments of order $t$ can be reconstructed from the low-dimensional moments 
\begin{equation*}
\mathbb{E}(Q_jX)^{s}, \quad j=1,\ldots,n, \quad |s|= t.
\end{equation*}
Thus, we aim to find $\{Q_j\}_{j=1}^n$, such that, for each $t\leq T$, \eqref{eq:generater} spans $\Hom_t(\R^d)$. Note that spanning sets in finite dimensions are also called frames. 

The most excessive dimension reduction corresponds to $k=1$. In this case, we observe that we only need to address $t=T$:
\begin{proposition}\label{pro:con Q}
Let $\{Q_j\}_{j=1}^n\subset \mathcal{V}_{1,d}$ and $x\in\R^d$ be a vector of unknowns. 
\begin{itemize}
\item[a)] 
If $\{( Q_j x)^t\}_{j=1}^n$ is a frame for $\Hom_t(\R^d)$, then $\{( Q_j x)^{t-1}\}_{j=1}^n$ is a frame for $\Hom_{t-1}(\R^d)$.
\item[b)] If $\{( Q_j x)^{t-1}\}_{j=1}^n$ is linearly independent in $\Hom_{t-1}(\R^d)$, then $\{( Q_j x)^{t}\}_{j=1}^n$ is linearly independent in $\Hom_{t}(\R^d)$.
\end{itemize}
\end{proposition}
\begin{proof}
a) Let $f$ be an arbitrary element in $\Hom_{t-1}(\R^d)$. There is $g\in \Hom_{t}(\R^d)$ such that its first partial derivative $\partial_1 g$ coincides with $f$. Since $\{( Q_j x)^t\}_{j=1}^n$ is a frame for $\Hom_t(\R^d)$, there are coefficients $\{c_j\}_{j=1}^n$ such that $g=\sum_{j=1}^n  c_j (Q_jx)^t$. Therefore, we obtain
\begin{equation*}
f(x) = \sum_{j=1}^n  c_j  (Q_je_1)  t(Q_jx)^{t-1},
\end{equation*}
which verifies part a).

b) Suppose that $0 = \sum_{j=1}^n  c_j (Q_jx)^{t}$. Applying all partial derivatives yield
\begin{equation*}
0 = \sum_{j=1}^n  c_j (Q_je_i) t(Q_jx)^{t-1},\qquad i=1,\ldots,d.
\end{equation*}
The linear independence assumption implies $c_j (Q_je_i)=0$, for $i=1,\ldots,d$, and, therefore, $c_j=0$, for $j=1,\ldots,n$, since $Q_j\neq 0$.
\end{proof}

Part a) of Proposition \ref{pro:con Q} tells us that if $\{( Q_j x)^t\}_{j=1}^n$ is a frame for $\Hom_t(\R^d)$, then 
\begin{equation}\label{eq:generater2}
\{(Q_jx)^{s} : j=1,\ldots,n,\; s\in\N^{k},\; |s|\leq t\}
\end{equation}
is a frame for $\Pol_t(\R^d)$. The proof directly shows that the first low-dimensional moments are sufficient to reconstruct the first high-dimensional moments. 

%Part b) of Proposition \ref{pro:con Q} tells us that any basis for $\Hom_{t-1}(\R^d)$ can be extended to a basis for $\Hom_t(\R^d)$. 

Next,  we provide a general construction recipe of $\{Q_j\}_{j=1}^n\subset \mathcal{V}_{1,d}$ that covers arbitrary $d$ and $t$. Note that the dimension of $\Hom_t(\R^d)$ is $\binom{t+d-1}{d-1}$:
\begin{theorem}\label{th:con Q}
Let $\{v_i\}_{i=1}^d$ be pairwise different positive real numbers, let $\{\alpha_j\}_{j=1}^{t+d-1}$ be pairwise different nonnegative integers, and let $V = (v_i^{\alpha_j})_{i,j}$ denote the associated $(t+d-1)\times d$-Vandermonde type matrix.  Suppose that the $\binom{t+d-1}{d-1}\times d$ matrix $Q$ is build from all minors of $V$ of order $d-1$. We denote the rows of $Q$ by $Q_1,\ldots,Q_n$, where $n=\binom{t+d-1}{d-1}$. Then $\{(Q_jx)^t\}_{j=1}^n$ is a basis for $\Hom_t(\R^d)$.
\end{theorem}
\begin{proof}
We expand $(Q_jx)^t$ by the multivariate binomial formula 
\begin{equation*}
(Q_jx)^t = \sum_{\alpha\in\N^d,\;|\alpha|=t} \binom{t}{\alpha} Q_j^\alpha x^\alpha.
\end{equation*}
The coefficients are put into the $j$-th row of a matrix $M_1\in \R^{n\times n}$, i.e.,  
\begin{equation*}
M_1 = \Big( \left(\begin{smallmatrix} t\\ \alpha
\end{smallmatrix}\right)
Q_j^\alpha \Big)_{j,\alpha} .
\end{equation*}
We must now check that $M_1$ is invertible. 

Dividing each column $\alpha$ by its respective binomial coefficient $\binom{t}{\alpha}$ yields the matrix $M_2=\big(Q_j^\alpha \big)_{j,\alpha} \in \R^{n\times n}$, and $M_1$ is invertible if and only if $M_2$ is. Let $c$ denote the product of all minors of order $d$ of $V$. It follows from \cite{Yaacov:2014fr} that 
\begin{equation*}
\det(M_2) = c^{d-1}.
\end{equation*}
The Vandermonde structure yields that $c\neq 0$, so that $M_2$ and hence $M_1$ is invertible. Thus, $\{(Q_jx)^t\}_{j=1}^n$ is indeed a basis for $\Hom_t(\R^d)$.
\end{proof}
Note that normalization of the rows of $Q$ in Theorem \ref{th:con Q} yields $\{Q_j\}_{j=1}^n\subset \mathcal{V}_{1,d}$, and $\{(Q_jx)^t\}_{j=1}^n$ is a basis for $\Hom_t(\R^d)$. Thus, for each $s\leq t$, $\{(Q_jx)^s\}_{j=1}^n$ is a frame for $\Hom_s(\R^d)$ according to Proposition \ref{pro:con Q}.

\section{Frames vs.~cubatures and a moment reconstruction formula}\label{sec:4}
So far, we have seen that reconstruction of high dimension moments from low dimensional ones is related to frames for $\Hom_t(\R^d)$. Next, we shall relate such frames to cubature points.  Let $\Hom_t(\mathbb{S}^{d-1})$ denote the space of homogeneous polynomials $\Hom_t(\R^d)$, but restricted to the sphere $\mathbb{S}^{d-1}$. For points $\{q_j\}_{j=1}^n\subset \mathbb{S}^{d-1}$ and weights $\{\omega_j\}_{j=1}^n\subset \R$, we say that $\{(q_j,\omega_j)\}_{j=1}^n$ is a cubature for $\Hom_t(\mathbb{S}^{d-1})$ if 
\begin{equation*}
\int_{\mathbb{S}^{d-1}} f(x) \mathrm d x = \sum_{j=1}^n \omega_j f(q_j),\quad \text{for all } f\in \Hom_t(\mathbb{S}^{d-1}),
\end{equation*}
where $dx$ denotes the standard measure on the sphere normalized to have mass one. 
It turns out that the frame property of $\{(Q_jx)^s\}_{j=1}^n$ is related to the concept of cubature points. 
\begin{theorem}\label{th:zwei}
Let $\{Q_j\}_{j=1}^n\subset \mathcal{V}_{1,d}$ and $x\in\R^d$ be a vector of unknowns. 
\begin{itemize}
\item[a)] If $\{(Q_jx)^t\}_{j=1}^n$ is a frame for $\Hom_t(\R^d)$, then there are weights $\{\omega_j\}_{j=1}^n\subset\R$, such that $\{(Q^\top_j,\omega_j)\}_{j=1}^n$ is a cubature for $\Hom_t(\mathbb{S}^{d-1})$.

\item[b)] If there are weights $\{\omega_j\}_{j=1}^n\subset\R$ such that $\{(Q^\top_j,\omega_j)\}_{j=1}^n$ is a cubature for $\Hom_{2t}(\mathbb{S}^{d-1})$, then $\{(Q_jx)^t\}_{j=1}^n$ is a frame for $\Hom_t(\R^d)$.
\end{itemize}
\end{theorem}
%In case $t=1$, Theorem \ref{th:zwei} is related to the concept of scalable frames introduced in \cite{Kutyniok:2013fk}. 
\begin{proof}
a) Since $\{(Q_jx)^t\}_{j=1}^n$ is a frame for $\Hom_t(\R^d)$, for each $a\in\mathbb{S}^{d-1}$, there are coefficients $\{c_j(a)\}_{j=1}^n\subset\R$ such that 
\begin{equation*}
(a^\top x)^t = \sum_{j=1}^nc_j(a)(Q_jx)^t.
\end{equation*}
Note that the mapping $a\mapsto c_j(a)$ can be chosen to be continuous, for each $j=1,\ldots,n$. Therefore, we derive
 \begin{align*}
 \int_{\mathbb{S}^{d-1}} (a^\top x)^t  \mathrm da = \sum_{j=1}^n (Q_jx)^t\int_{\mathbb{S}^{d-1}} c_j(a) \mathrm da =  \sum_{j=1}^n (Q_jx)^t\omega_j ,
 \end{align*}
 with $\omega_j=\int_{\mathbb{S}^{d-1}} c_j(a)  \mathrm da$. Since the above equality holds for all $x\in\R^d$, 
 $\{(Q^\top_j,\omega_j)\}_{j=1}^n$ is a cubature for $\Hom_t(\mathbb{S}^{d-1})$.
 
 b) Note that $\Hom_t(\mathbb{S}^{d-1})$ is a reproducing kernel Hilbert space and let us denote its reproducing kernel with respect to the standard inner product by $K_t$. For now, we restrict $x$ to the sphere and let $a\in\mathbb{S}^{d-1}$ as well. The reproducing property yields 
 \begin{align*}
(a^\top x)^t &= \int_{\mathbb{S}^{d-1}} (z^\top x)^t K_t(z,a)\mathrm dz.   \\
\intertext{The mapping $z\mapsto (z^\top x)^t K_t(z,a)$ is contained in $\Hom_{2t}(\mathbb{S}^{d-1})$, so that the cubature property yields
}
(a^\top x)^t  &= \sum_{j=1}^n \omega_j(Q_jx)^tK_t(Q^\top_j,a)
  = \sum_{j=1}^n (Q_jx)^t c_j(a),
 \end{align*}
where $c_j(a)=\omega_jK_t(Q^\top_j,a)$. A homogeneity argument concludes the proof. 
 \end{proof}
Note that the degree of the homogeneous polynomials in Part b) of Theorem \ref{th:zwei} is not the same but $2t$ for the cubatures and $t$ for the frame. The degree $2t$ is due to multiplication of two homogeneous polynomials of degree $t$, which is not just an artifact of the proof.  There are indeed cubatures for $\Hom_{t}(\mathbb{S}^{d-1})$, whose cardinality is lower than the dimension of $\Hom_{t}(\R^d)$.

%The linear span of the product $\Hom_{t}(\mathbb{S}^{d-1})\cdot\Hom_{t}(\mathbb{S}^{d-1})$ is $\Hom_{2t}(\mathbb{S}^{d-1})$. 

In fact, Theorem \ref{th:zwei} holds in much more generality in suitable finite dimensional reproducing kernel Hilbert spaces. Let $(\Omega,\sigma)$ be a finite measure space and let $\mathcal{F}$ be a linear subspace of continuous functions in $L_2(\Omega,\sigma)$. For points $\{q_j\}_{j=1}^n\subset \Omega$ and weights $\{\omega_j\}_{j=1}^n\subset \R$, we say that $\{(q_j,\omega_j)\}_{j=1}^n$ is a cubature for $\mathcal{F}$ if 
\begin{equation*}
\int_{\Omega} f(x) \mathrm d \sigma(x) = \sum_{j=1}^n \omega_j f(q_j),\quad \text{for all } f\in \mathcal{F}.
\end{equation*}  
%The reproducing kernel of $\mathcal{F}$ with respect to the $L_2$ inner product is denoted by $K_\mathcal{F}$. 
The following result generalizes Theorem \ref{th:zwei}:
\begin{proposition}\label{prop:all weights}
Let $K:\Omega\times\Omega\rightarrow\R$ be a symmetric kernel that linearly generates $\mathcal{F}$, i.e., $K(x,y)=K(y,x)$, for $x,y\in\Omega$, and 
\begin{equation}\label{eq:span K and so}
\mathcal{F} = \spann\{K(a,\cdot) : a\in\Omega\}.
\end{equation}
For $\{q_j\}_{j=1}^n\subset \Omega$, the following holds:
\begin{itemize}
\item[a)] If $\{K(q_j,\cdot)\}_{j=1}^n$ is a frame for $\mathcal{F}$, then there are weights $\{\omega_j\}_{j=1}^n\subset\R$, such that $\{(q_j,\omega_j)\}_{j=1}^n$ is a cubature for $\mathcal{F}$.

\item[b)]
If there are weights $\{\omega_j\}_{j=1}^n\subset\R$ such that $\{(q_j,\omega_j)\}_{j=1}^n$ is a cubature for the linear span of $\mathcal{F} \cdot \mathcal{F}$, then $\{K(q_j,\cdot)\}_{j=1}^n$ is a frame for $\mathcal{F}$.
\end{itemize}
\end{proposition}
The proof of Proposition \ref{prop:all weights} is structurally the same as for Theorem \ref{th:zwei} with $K(x,y)=(x^\top y)^t$ and $\mathcal{F}=\Hom_t(\mathbb{S}^{d-1})$, so we omit the details. 
%\begin{proof}
%a) Since $\{K(q_j,\cdot)\}_{j=1}^n$ is a frame for $\mathcal{F}$, there are coefficients $\{c_j(a)\}_{j=1}^n\subset\R$, for each $a\in\Omega$, such that 
%\begin{equation*}
%K(a,\cdot) = \sum_{j=1}^n c_j(a)K(q_j,\cdot).
%\end{equation*}
%As before, we deduce
% \begin{align*}
% \int_{\Omega} K(a,\cdot) \mathrm d\sigma(a) &= \sum_{j=1}^n K(q_j,\cdot) \int_{\Omega} c_j(a) \mathrm d\sigma(a) \\
% &=  \sum_{j=1}^n K(q_j,\cdot)\omega_j ,
% \end{align*}
% with $\omega_j=\int_{\Omega} c_j(a)\mathrm d\sigma(a)$. Thus, \eqref{eq:span K and so} yields that $\{(q_j,\omega_j)\}_{j=1}^n$ is a cubature for $\mathcal{F}$.
%
%b) For $a\in\Omega$, the reproducing property yields 
% \begin{align*}
%K(a,x) &= \int_{\Omega} K(z,x) K_\mathcal{F}(z,a) \mathrm d\sigma(z).   \\
%\intertext{The mapping $z\mapsto K(z,x)  K_\mathcal{F}(z,a)$ is contained in $\mathcal{F} \cdot \mathcal{F}$, so that the cubature property yields
%}
%K(a,x) &  = \sum_{j=1}^n \omega_jK(q_j,x)K_\mathcal{F}(q_j,a)\\
%&  = \sum_{j=1}^n K(q_j,x) c_j(a),
% \end{align*}
%where $c_j(a)=\omega_jK_\mathcal{F}(q_j,a)$.
%\end{proof}
%Note that Theorem \ref{th:zwei} follows from Proposition \ref{prop:all weights} by setting $\mathcal{F}=\Hom_t(\mathbb{S}^{d-1})$ and $K(x,y)=(x^\top y)^t$, for $x,y\in\mathbb{S}^{d-1}$, with an additional homogeneity argument to switch between $\mathbb{S}^{d-1}$ and $\R^d$. 
\begin{remark}
Part b) of Proposition \ref{prop:all weights} implies $n\geq \dim(\mathcal{F})$. Analoguous results in \cite{Harpe:2005fk}, for instance, are restricted to positive weights. 
\end{remark}

Note that $Q\in\mathcal{V}_{1,d}$ if and only if $Q^\top Q\in\mathcal{G}_{1,d}$. Moreover, the kernel 
\begin{equation*}
K_{t,1}:\mathcal{G}_{1,d}\times \mathcal{G}_{1,d}\rightarrow \R,\quad (P,R)\mapsto \trace(PR)^t
\end{equation*}
 linearly generates $\Hom_t(\G_{1,d})$. Here, the space of homogeneous polynomials of degree $t$ on $\G_{k,d}$ is defined by restrictions of homogeneous polynomials of degree $t$ on $\R^{d\times d}_{\sym}$. Moreover, $\G_{k,d}$ is naturally endowed with an orthogonal invariant probability measure $\sigma_{k,d}$. For $x\in\mathbb{S}^{d-1}$, 
\begin{equation}\label{eq:sphere = G}
(Q x)^{2t} = K_{t,1}(Q^\top Q , xx^\top),
\end{equation}
 so that $\Hom_{2t}(\mathbb{S}^{d-1})$ corresponds to $\Hom_t(\G_{1,d})$. According to \eqref{eq:sphere = G} we deduce that, for $\{Q_j\}_{j=1}^n\subset\mathcal{V}_{1,d}$, the set $\{(Q_jx)^{2t}\}_{j=1}^n$ is a frame for $\Hom_{2t}(\R^d)$ if and only if $\{K_{t,1}(Q_j^\top Q_j,\cdot)\}_{j=1}^n$ is a frame for $\Hom_t(\G_{1,d})$. Similarly, $\{(Q_j^\top,\omega_j)\}_{j=1}^n$ is a cubature for $\Hom_{2t}(\mathbb{S}^{d-1})$ if and only if $\{(Q_j^\top Q_j,\omega_j)\}_{j=1}^n$ is a cubature for $\Hom_{t}(\G_{1,d})$. Therefore, we can switch to the Grassmannian setting to formulate the following moment reconstruction result:
 \begin{corollary}[\cite{Graf:2014qd}]\label{cor:zweiter Versuch}
For $r\in\N^d$ with $|r|\leq t\leq d$, there are coefficients $a^r_s\in\R$, $s\in\N^d$, $|s|=|r|$, such that 
if $\{(P_j,\omega_j)\}_{j=1}^n$ is a cubature for $\Hom_{t}(\mathcal{G}_{1,d})$, then any random vector $X\in \R^d$ satisfies
\begin{equation}\label{eq:fundament 2}
\mathbb{E}X^r =\sum_{j=1}^n \omega_j \sum_{s\in\N^d,\;|s|= |r|}  a^r_s \mathbb{E} (P_jX)^s.
%\mathbb{E}X^\alpha = \sum_{s=1}^t \sum_{i=1}^m f_{s,i}\sum_{j=1}^n\omega_j \mathbb{E}\langle P_jX,y_i\rangle^s.
\end{equation}
\end{corollary}
For $Q_j\in\mathcal{V}_{1,d}$ and $P_j\in\mathcal{G}_{1,d}$ with $P_j=Q_j^\top Q_j$, one switches between the moments of $Q_jX$ and $P_jX$ by the formula
\begin{equation*}
(P_jX)^s = Q_j^s (Q_jX)^{|s|}, \quad s\in\N^d.
\end{equation*}
It may depend on the context when $P_jX$ or $Q_jX$ is preferred.

 \section{Cubatures in Grassmannians}
Proposition \ref{prop:all weights} connects frames and cubatures beyond $\G_{1,d}$ and applies to the general Grassmannians $\G_{k,d}$ by $\Omega=\G_{k,d}$, $\mathcal{F}=\Hom_t(\G_{k,d})$, and the kernel $K=K_{t,k}$ given by
 \begin{equation*}
 K_{t,k}:\G_{k,d}\times \G_{k,d}\rightarrow \R,\quad (P,R)\mapsto \trace(PR)^t,
 \end{equation*}
 cf.~\cite{Graf:2014qd,Ehler:2014zl}. 
In the following sections, we shall provide further examples for the usefulness of Grassmannian cubatures beyond moment reconstruction. To begin with, we address the issue of constructing cubatures.

\subsection{Numerical construction of cubatures}\label{subsec:5.1}
Cubatures on Grassmannians with constant weights are constructed in \cite{Bachoc:2002aa} from group orbits. There is also a simple method to numerically compute cubature points by some minimization method as we shall outline next. The $t$-fusion frame potential for points $\{P_j\}_{j=1}^n\subset\G_{k,d}$ and weights $\{\omega_j\}_{j=1}^n\subset\R$ is 
\begin{equation*}
\FFP(\{(P_j,\omega_j)\}_{j=1}^n,t):= \sum_{i,j=1}^{n} \omega_j\omega_i\tr(P_{i}P_{j})^{t}.
\end{equation*}
Assuming that $\sum_{j=1}^n\omega_j=1$, the fusion frame potential is lower bounded by  
 \begin{equation}\label{eq:low boun}
\FFP(\{(P_j,\omega_j)\}_{j=1}^n,t)\geq \int_{\G_{k,d}}\int_{\G_{k,d}} \tr(PR)^{t} \mathrm d \sigma_{k,d}(P) \mathrm
  d \sigma_{k,d}(R),
 \end{equation}
cf.~\cite{Breger:2016vn} and also \cite{Bachoc:2010aa}. Since the constant functions are contained in $\Hom_t(\G_{k,d})$, any cubature must satisfy $\sum_{j=1}^n\omega_j=1$. 
\begin{theorem}[\cite{Breger:2016vn,Bachoc:2010aa}]
If $\sum_{j=1}^n\omega_j=1$ and \eqref{eq:low boun}  holds with equality, then $\{(P_j,\omega_j)\}_{j=1}^n$ is a cubature for $\Hom_t(\G_{k,d})$. 
\end{theorem}
 In order to check for equality in \eqref{eq:low boun}, we require a more explicit expression for the right-hand-side. In fact, it holds
\begin{equation*}
\int_{\G_{k,d}}\int_{\G_{k,d}} \tr(PR)^{t} \mathrm d \sigma_{k,d}(P) \mathrm
  d \sigma_{k,d}(R) = \sum_{\substack{|\pi| = t,\\
      \ell(\pi) \le d/2}
  } \frac{C_{\pi}^2(I_{k})}{C_{\pi}(I_{d})},
\end{equation*}
where $I_d$ denotes the $d\times d$ identity matrix and $\pi$ is an integer partition of $t$ with $\ell(\pi)$ being the number of nonzero parts, and $C_\pi$ are the zonal polynomials, cf.~\cite{Chikuse:2003aa,Muirhead:1982fk,Gross:1987bf}. Evaluation of $C_\pi$ at $I_k$ and $I_d$, respectively, yields
\begin{equation*}
  C_{\pi}(I_{d}) = 2^{|\pi|}|\pi|!\big( \tfrac d2 \big)_{\pi}
  \prod_{1\le i<j \le \ell(\pi)} (2 \pi_{i} - 2\pi_{j} - i+ j) /
  \prod_{i=1}^{\ell(\pi)} (2\pi_{i} + \ell(\pi) - i)!,
\end{equation*}
cf.~\cite{Dumitriu:2007ax}. Here, $(a)_{\pi}$ denotes the generalized hypergeometric coefficient given by
\begin{equation}
\label{eq:hypergeomcoeff}
(a)_{\pi}:= \prod_{i=1}^{\ell(\pi)}\big(a-\tfrac12(i-1)\big)_{\pi_{i}},
\qquad (a)_{s} := a(a+1)\dots(a+s-1).
\end{equation}

Fixing the weights $\{\omega_j\}_{j=1}^n\subset\R$, say $\omega_j=1/n$, for $j=1,\ldots,n$, we can now aim to numerically minimize the $t$-fusion frame potential $\FFP(\{(P_j,\omega_j)\}_{j=1}^n,t)$ over all sets of $n$ points $\{P_j\}_{j=1}^n\subset\G_{k,d}$ and check for equality in \eqref{eq:low boun}, where the right-hand-side can be computed explicitly. See \cite{Breger:2016rc,Breger:2016vn} for successful minimizations in $\G_{2,4}$.

\subsection{Cubatures for approximation of integrals}\label{sec:5.2}
Cubature points enable us to replace integrals over polynomials by finite sums. We now aim to go beyond polynomials and keep track of the integration error. Without loss of generality, we assume $k\leq \frac{d}{2}$ throughout since $\mathcal{G}_{d-k,d}$ can be identified with $\mathcal{G}_{k,d}$. 

The eigenfunctions $\{\varphi_\pi\}_{\ell(\pi)\leq k}$ of the Laplace-Beltrami operator $\Delta$ on $\mathcal{G}_{k,d}$ are an orthonormal basis for $L_2(\G_{k,d})$ and are naturally indexed by integer partitions $\pi$ of length at most $k$. Let $\{-\lambda_\pi\}_{\ell(\pi)\leq k}$ be the corresponding eigenvalues, i.e., 
\begin{equation}\label{eq:eig values explicit} 
\lambda_\pi = 2|\pi|d+4\sum_{i=1}^k \pi_i(\pi_i-i),
\end{equation}
cf.~\cite[Theorem 13.2]{James:1974aa}. Without loss of generality, we choose each $\varphi_\pi$ to be real-valued, in particular, $\varphi_{(0)}\equiv 1$. Essentially following \cite{Brandolini:2014oz,Mhaskar:2010kx}, we formally define $(I-\Delta)^{s/2} f$ to be the distribution on $\mathcal{G}_{k,d}$, such that 
\begin{equation*}
\langle (I-\Delta)^{s/2} f,\varphi_\pi\rangle = (1+\lambda_\pi)^{s/2}\langle f,\varphi_\pi\rangle, \quad\text{ for all }\ell(\pi)\leq k.
\end{equation*}
The Bessel potential space $H^s_p(\mathcal{G}_{k,d})$, for $1\leq p\leq \infty$ and $s \ge 0$, is
\begin{align*}
H^s_p(\mathcal{G}_{k,d})&:=\{f\in L_p(\mathcal{G}_{k,d}) :  \| f\|_{H^s_p}<\infty\},\quad\text{where}\\
\| f\|_{H^s_p} &:= \| (I-\Delta)^{s/2} f  \|_{L_p},
\end{align*}
i.e., $f\in H^s_p(\mathcal{G}_{k,d})$ if and only if $f\in L_p(\mathcal{G}_{k,d})$ and $(I-\Delta)^s f\in L_p(\mathcal{G}_{k,d})$. 

The expected worst case error of integration in Bessel potential spaces of $n$ independent random points endowed with constant weights is of the order $n^{-\frac{1}{2}}$: 
\begin{proposition}[\cite{Graf:2013zl,Nowak:2010rr,Brauchart:fk,Breger:2016vn}]\label{th:random inde}
For $s>k(d-k)/2$, suppose $P_1,\ldots,P_n$ are random points on $\mathcal{G}_{k,d}$, independently identically distributed according to $\sigma_{k,d}$ then it holds
\begin{equation*}
\sqrt{\mathbb{E} \Big[\sup_{\substack{f\in H^s_2(\mathcal{G}_{k,d})\\ \|f\|_{H^s_2}\leq 1}}  \Big|\int_{\mathcal{G}_{k,d}} f(P)\mathrm d\sigma_{k,d}(P) - \frac{1}{n}\sum_{j=1}^n f(P_j)  \Big|^2\Big]}  = cn^{-\frac{1}{2}}
\end{equation*}
with $c^2=\sum_{1\leq \ell(\pi)\leq k} (1+\lambda_\pi)^{-s}$.
% \begin{equation*}
% c=\sum_{0<l(\pi)\leq k} (1+\lambda_\pi)^{-s}J^{\frac k2, \frac d 2}_\pi(1,\ldots,1).
% \end{equation*}
\end{proposition}
%We need a closer look at cubature points, for which we need the space of
%\emph{diffusion polynomials} of degree at most $t$, defined by
%\begin{equation}\label{eq:poly diff}
%\Pi_t:=\spann\{\varphi_\ell:\lambda_\ell\leq  t^2 \},
%\end{equation}
%see \cite{Mhaskar:2010kx} and references therein.   
%For $\{P_j\}_{j=1}^n\subset\mathcal{G}_{k,d}$ and positive weights
%  $\{\omega_j\}_{j=1}^n$, we say that $\{(P_j,\omega_j)\}_{j=1}^n$ is a
%  \emph{cubature for $\Pi_t$} if 
%    \begin{equation}\label{eq:def cub 00}
%    \int_{\mathcal{G}_{k,d}} f(P)\mathrm d\mu_{k,d}(P) = \sum_{j=1}^n \omega_j f(P_j),\quad \text{for all } f\in  \Pi_t.
%  \end{equation}
% Note that $\Pi_t$ and $\Pol_t(\G_{k,d})$ do not differ too much:
%  \begin{theorem}[\cite{Breger:2016vn}]\label{pro:1}
%Polynomials and diffusion polynomials on the Grassmannian $\G_{k,d}$ satisfy the relation
%\begin{equation*}
%\Pi_{s(t+1)-\epsilon} \subset  \Pol_{t}(\G_{k,d}) \subset \Pi_{\sqrt{4t^2+2t(d-2)}}, \quad \text{for all $0<\epsilon<2s(t+1)$},
%\end{equation*}
%where $s(t)=\sqrt{\lceil\tfrac 4k t^{2} + 2t(d - k -1)\rceil}$.
%\end{theorem}
%  
  The following result follows from \cite[Theorem 2.12]{Brandolini:2014oz}:
\begin{theorem}[\cite{Brandolini:2014oz}]\label{th:brandolini}
Let $s>k(d-k)/p$. Any sequence of cubatures $\{(P^{(t)}_j,\omega^{(t)}_j)\}_{j=1}^{n_t}$ with nonnegative weights for $\Hom_t(\G_{k,d})$, $t=1,2,\ldots,$ satisfies, for $f\in H^s_p(\mathcal{G}_{k,d})$,
\begin{equation*}
\Big|\int_{\mathcal{G}_{k,d}} f(P)\mathrm d\sigma_{k,d}(P)-\sum_{j=1}^{n_t} \omega^t_j f(P^t_j)\Big| \lesssim  t^{-s}\|f\|_{H^s_p}.
\end{equation*}
\end{theorem}
\begin{remark}\label{rem:111}
For any $t=1,2,\ldots$, there exist cubatures $\{(P^{(t)}_j,\omega^{(t)}_j)\}_{j=1}^{n_t}$ for $\Hom_t(\G_{k,d})$ with positive weights such that 
\begin{equation}\label{eq:asym cov}
n_t \asymp t^{k(d-k)},\quad t=1,2,\ldots,
\end{equation}
 cf.~\cite{Harpe:2005fk}. Grassmannian $t$-designs are cubatures for $\Hom_t(\G_{k,d})$ with constant weights $\omega_j=1/n$, for $j=1,\ldots,n$, and there exist Grassmannian $t$-designs that satisfy the asymptotics \eqref{eq:asym cov}, cf.~\cite{Etayo:2016qq}.
 \end{remark}
%
%Weyl's estimates on the spectrum of an elliptic operator yield 
%\begin{equation*}
%\dim(\Pi_t)\asymp t^{k(d-k)},
%\end{equation*}
%cf.~\cite[Theorem 17.5.3]{Hormander:1983gf}. This implies that any sequence of cubatures $\{(P^t_j,\omega^t_j)\}_{j=1}^{n_t}$ of strength $t$, respectively, must obey 
%$
%n_t\gtrsim t^{k(d-k)}
%$ 
% asymptotically in $t$, cf.~\cite{Harpe:2005fk}. There are indeed sequences of cubatures $\{(P^t_j,\omega^t_j)\}_{j=1}^{n_t}$ of strength $t$, respectively, satisfying 
% \begin{equation}\label{eq:few cubs2}
%n_t\asymp t^{k(d-k)},
%\end{equation}
%cf.~\cite{Harpe:2005fk}. 
If \eqref{eq:asym cov} is satisfied, then Theorem \ref{th:brandolini} leads to
\begin{equation}\label{eq:general cub fomula}
\Big|\int_{\mathcal{G}_{k,d}} f(P)\mathrm d\mu_{k,d}(P)-\sum_{j=1}^{n_t} \omega^t_j f(P^t_j)\Big| \lesssim  n_t^{-\frac{s}{k(d-k)}}\|f\|_{H^s_p}.
\end{equation}
The condition $s>k(d-k)/p$, for $p=2$, in Theorem \ref{eq:asym cov} then yields that cubature points do better than the random points in Proposition \ref{th:random inde}. 
Given any sequence of points of cardinality $n_t$, the rate $n_t^{-\frac{s}{k(d-k)}}$ cannot be improved, cf.~\cite{Brandolini:2014oz}.

%\begin{remark}
%Cubature points $\{P_j\}_{j=1}^n$ for $\Pi_t$ with constant weights $\omega_j=\frac{1}{n}$ are called $t$-designs. 
%  For all $t=1,2,\ldots$, there exist $t$-designs, cf.~\cite{Seymour:1984bh}.
%The results in \cite{Bondarenko:2011kx} imply that there are $t$-designs satisfying \eqref{eq:few cubs2} provided that $k=1$. However, for $2\leq k\leq \frac{d}{2}$, it is still an open problem if the asymptotics
%  \eqref{eq:few cubs2} can be achieved by $t$-designs in place of cubatures.
%\end{remark}

\subsection{Cubatures for function approximation}\label{sec:5.3}
The basic idea for applying cubature points in function approximation is quite simple. 
%Assume we have an orthogonal basis for $L_2(\G_{k,d})$. We shall work with eigenfunctions $\{\varphi_\pi\}_{\ell(\pi)\leq k}$ of the Laplace-Beltrami operator on $\G_{k,d}$, which are indeed naturally indexed by integer partitions $\pi$ with $\ell(\pi)\leq k$. Note that we suppose $k\leq d/2$ from here on. 
The standard expansion of any $f\in L_2(\G_{k,d})$ in the orthogonal basis $\{\varphi_\pi\}_{\ell(\pi)\leq k}$ yields  
\begin{equation}\label{eq:approx 1}
f  = \sum_{\ell(\pi)\leq k}  \langle f, \varphi_\pi\rangle \varphi_\pi \approx \sum_{\substack{|\pi|\leq t\\ \ell(\pi)\leq k }}\langle f, \varphi_\pi\rangle \varphi_\pi,
\end{equation}
where the approximation is simply derived by truncating the infinite series at $|\pi|\leq t$. The inner product $\langle f, \varphi_\pi\rangle$ is an integral that we approximate by the concept of cubatures, i.e., the error for approximating the integral by a finite sum is steered by \eqref{eq:general cub fomula}:
\begin{align}
\sum_{\substack{|\pi|\leq t\\ \ell(\pi)\leq k }}\langle f, \varphi_\pi\rangle \varphi_\pi& = \sum_{\substack{|\pi|\leq t\\ \ell(\pi)\leq k }} \int_{\G_k,d} f(P) \varphi_i(P)\mathrm d\sigma_{k,d}(P)  \varphi_i \nonumber\\
& \approx  \sum_{\substack{|\pi|\leq t\\ \ell(\pi)\leq k }}\sum_{j=1}^n \omega_j f(P_j) \varphi_\pi(P_j)  \varphi_\pi \label{eq:second apporx}\\
&  =  \sum_{j=1}^n \omega_j f(P_j) \sum_{\substack{|\pi|\leq t\\ \ell(\pi)\leq k }} \varphi_\pi(P_j)  \varphi_\pi.\nonumber
\end{align}
If we define $K_t(P,R)=\sum_{\substack{|\pi|\leq t\\ \ell(\pi)\leq k }} \varphi_\pi(P)  \varphi_\pi(R)$, then we arrive at the approximation 
\begin{equation}\label{eq:approx all}
f\approx  \sum_{j=1}^n \omega_j f(P_j) K_t(P_j,\cdot).
\end{equation}
The right-hand-side of \eqref{eq:approx all} is composed by two separate approximations, truncation of the series \eqref{eq:approx 1} and the approximation of the integral via cubatures \eqref{eq:second apporx}. To obtain suitable error rates, it turns out that we better replace the sharp truncation by a smoothed version, i.e., we define the kernel $K_t$ on $\G_{k,d} \times \G_{k,d}$ by
\begin{equation}\label{eq:sigma again and new}
  K_t(P,Q)= \sum_{\ell(\pi)\leq k} h(t^{-2}\lambda_\pi) \varphi_\pi(P)\varphi_\ell(Q),
\end{equation}
where $h:\R_{\geq 0}\rightarrow\R$ is an infinitely often differentiable and
nonincreasing function with $h(x)=1$, for $x\le 1/2$, and $h(x)=0$, for
$x\ge 1$. 
%\begin{equation}\label{eq:poly diff}
%\Pi_t:=\spann\{\varphi_\ell:\lambda_\ell\leq  t^2 \},
%\end{equation}
The smoothed series truncation becomes the expression 
\begin{equation}\label{eq:approx sigma}
\sigma_{t}(f) := \int_{\mathcal{G}_{k,d}} f(P) K_t(P,\cdot) \mathrm d\sigma_{k,d}(P),
\end{equation}
and $\sigma_t(f)$ approximates $f$ with an error rate that matches the ones in Theorem \ref{th:brandolini}:
\begin{theorem}[\cite{Mhaskar:2010kx}]\label{th:0}
If $f\in H^s_p(\mathcal{G}_{k,d})$, then 
\begin{equation*}
 \|f - \sigma_t(f)\|_{L_p} \lesssim t^{-s} \|f\|_{H^s_p}.
\end{equation*}
\end{theorem}
To approximate $f$ from finitely many samples, we combine the smoothed truncation with cubature points to replace the integral by a finite sum. For sample points $\{P_j\}_{j=1}^{n}\subset\mathcal{G}_{k,d}$ and weights $\{\omega_j\}_{j=1}^{n}$, we define
\begin{equation}\label{eq:def sigma tilde etc}
\sigma_t(f,\{(P_j,\omega_j)\}_{j=1}^n):=\sum_{j=1}^n \omega_j f(P_j)K_t(P_j,\cdot)
\end{equation}
which coincides with \eqref{eq:approx all} but the kernel $K_t$ is from \eqref{eq:sigma again and new}. 
Note that we must now consider functions $f$ in Bessel potential spaces, for which point evaluation makes sense. Note that  $\sigma_{r(t)}(f,\{(P_j,\omega_j)\}_{j=1}^n)$ is contained in $\Hom_t(\G_{k,d})$ for $r(t)=\sqrt{\lceil \frac{4}{k}t^2\rceil}$, cf.~\cite[Theorem 5]{Breger:2016vn} for a slightly larger function $r$. The following approximation is a consequence of \cite[Proposition 5.3]{Mhaskar:2010kx}:
\begin{theorem}[\cite{Breger:2016vn}]\label{th:m 0}
If $\{(P^{(t)}_j,\omega^{(t)}_j)\}_{j=1}^{n_t}$ is a sequence of cubatures with nonnegative weights for $\Hom_{2t}(\G_{k,d})$, $t=1,2,\ldots$, then, for $f\in H^s_\infty(\mathcal{G}_{k,d})$, 
\begin{equation}\label{eq:error approx finale}
\|f-\sigma_{r(t)}(f,\{(P^{(t)}_j,\omega^{(t)}_j)\}_{j=1}^{n_t})\|_{L_\infty} \lesssim t^{-s} (\|f\|_{L_\infty}+ \|f\|_{H^s_\infty}),
\end{equation}
where $r(t)=\sqrt{\lceil \frac{4}{k}t^2\rceil}$.
\end{theorem}
For cubatures satisfying $n_t\asymp t^{k(d-k)}$, see Remark \ref{rem:111}, \eqref{eq:error approx finale} becomes
\begin{equation}\label{eq:rate in n finale}
\|f-\sigma_{r(t)}(f,\{(P^t_j,\omega^t_j)\}_{j=1}^{n_t})\|_{L_\infty} \lesssim n_t^{-\frac{s}{k(d-k)}} (\|f\|_{L_\infty}+ \|f\|_{H^s(L_\infty)}),
\end{equation}
so that we obtain error rates similiar to \eqref{eq:general cub fomula}.

\subsection{Cubatures as efficient coverings}\label{Sec:5.4}
We have seen in the previous sections that cubatures relate to the approximation of integrals and are also useful to approximate functions from samples. Intuitively, good samplings for approximation need to cover the underlying space sufficiently well. Indeed, we shall connect cubatures with asymptotically optimal coverings.

Given any finite collection of points
$\{P_j\}_{j=1}^n\subset \G_{k,d}$, we define the \emph{covering radius} $\rho$ by
\begin{equation*}
\rho:=\rho(\{P_j\}_{j=1}^n):=\sup_{P\in \G_{k,d}} \min_{1\leq j\leq n} \|P-P_j\|,
\end{equation*}
where $\|\cdot\|$ denotes the Frobenius norm on the space of symmetric matrices. Let $B_r(P)$ denote the closed ball of radius $r$ centered at $P\in\G_{k,d}$. Since 
\begin{equation*}
\G_{k,d}=\bigcup_{j=1}^n B_\rho(P_j)
\end{equation*}
 and $\sigma_{k,d}(B_r(P)) \asymp r^{k(d-k)}$, for $P \in \G_{k,d}$ with $0<r\leq 1$, we deduce
\begin{equation*}
1=\sigma_{k,d}(\G_{k,d})\leq \sum_{j=1}^n \sigma_{k,d}(B_\rho(x_j))\lesssim n\rho^{k(d-k)},
\end{equation*}
which leads to the lower bound $n^{-\frac{1}{k(d-k)}}\lesssim \rho$. Point sequences in $\G_{k,d}$ that match this lower bound asymptotically in $n$ are referred to as asymptotically optimal coverings.  

%The cubatures $\{(P^{(t)}_j,\omega^{(t)}_j)\}_{j=1}^{n_t}$ for $\Hom_t(\G_{k,d})$ satisfying \eqref{eq:asym cov} in Remark \ref{rem:111}, i.e., $n_t \asymp t^{k(d-k)}$, for $t=1,2,\ldots$, cover optimally:
\begin{theorem}[\cite{Breger:2016rc}]\label{th:basis}
Any cubature sequence $\{(P^{(t)}_j,\omega^{(t)}_j)\}_{j=1}^{n_t}$ for $\Hom_t(\G_{k,d})$ with positive weights satisfying \eqref{eq:asym cov} is covering asymptotically optimal, i.e., its covering radius $\rho^{(t)}$ satisfies $n^{-\frac{1}{k(d-k)}}\asymp\rho^{(t)}$.
\end{theorem}
Theorems \ref{th:basis} and \ref{th:m 0} with \eqref{eq:rate in n finale} both reflect the intuition that cubature points satisfying \eqref{eq:asym cov} must be somewhat efficiently distributed on the underlying space.

\subsection{Cubatures for phase retrieval}\label{sec:5.5}
To reflect the versatility of Grassmannian cubatures, we now briefly discuss their use in phase retrieval. The problem of reconstructing vectors from phaseless magnitude measurements has attracted great attention in the recent literature, \cite{Candes:uq,Candes:2012fk,Conca:2013ud,Bandeira:2013fk,Demanet:2012uq} to mention only few. For $x\in\R^d$, the mapping 
\begin{equation*}
\hat{x}:\G_{k,d}\rightarrow\R, \quad P\mapsto \|Px\|^2
\end{equation*}
 is a homogeneous polynomial of degree $2$, hence, contained in $\Hom_2(\G_{k,d})$. Notice that $\hat{x}$ is in a one to one correspondence with the rank one matrix $xx^\top$ since 
 \begin{equation*}
 \hat{x}(P)=x^\top P x = \trace(Pxx^\top).
 \end{equation*}
 The problem of reconstructing  $xx^\top$ from finitely many samples $\{\hat{x}(P_j)\}_{j=1}^n$, where $\{P_j\}_{j=1}^n\subset\G_{k,d}$, is known as the phase retrieval problem. Most publication deal with $k=1$. For $k>1$, we refer to \cite{Bachoc:2012fk,Edidin:2017qp,Ehler:2014qc,Ehler:2013kx,Bahmanpour:bf} and references therein.
 
If there are weights $\{\omega_j\}_{j=1}^n$ such that $\{(P_j,\omega_j)\}_{j=1}^n$ is a cubature for $\Hom_2(\G_{k,d})$, then $xx^\top$ can be directly reconstructed via the closed formula
\begin{equation}\label{eq:recon Px}
xx^\top = \frac{d}{k}\sum_{j=1}^n \omega_j\hat{x}(P_j)\big[\frac{1}{\alpha }\sum_{j=1}^n \omega_j \hat{x}(P_j)P_{j} - \frac{\beta}{\alpha }I_d\big],
\end{equation}
where $\alpha=\frac{2k(d-k)}{d(d+2)(d-1)}$ and $
\beta=\frac{k(kd+k-2)}{d(d+2)(d-1)}$, cf.~\cite{Bachoc:2012fk}. However, cubatures for $\Hom_2(\G_{k,d})$ must have at least $d(d+1)/2$ many points. Thus, the number of samples grows quadratic with the ambient dimension $d$. We are seeking reconstruction from fewer samples at the expense of replacing the closed reconstruction formula with a feasibility problem of a semidefinite program. We consider the problem 
\begin{equation}\label{eq:convex II}
\text{find } A\in\R^{d\times d}_{\succeq 0},  \quad\text{subject to}\quad   \trace(P_jA) = \hat{x}(P_j), \; j=1,\ldots,n,
\end{equation}
where $\R^{d\times d}_{\succeq 0}$ denotes the symmetric, positive semidefinite matrices in $\R^{d\times d}$. 

\begin{theorem}[\cite{Bachoc:2012fk}]\label{th:finale!}
There are constants $c_1,c_2>0$ such that, if $n\geq c_1d$ and $\{P_j\}_{j=1}^n\subset\mathcal{G}_{k,d}$ are chosen independently identically distributed according to $\sigma_{k,d}$,  then the matrix $xx^\top$ is the unique solution to \eqref{eq:convex II} with probability $1-e^{-c_2 n}$, for all $x\in\R^d$. 
\end{theorem}
This theorem generalized results in \cite{Candes:uq,Candes:2012fk} from $k=1$ to $k\geq 1$. When the projectors $\{P_j\}_{j=1}^n$ are sampled from the idealized perfect cubature $\sigma_{k,d}$, then the number of samples needed grows linearly with $d$. Next, we shall find a balance between the deterministic cubatures required for \eqref{eq:recon Px} and the full randomness invoked by $\sigma_{k,d}$ used in Theorem \ref{th:finale!}. 

From here on, we suppose that the length $\|x\|$ is known to us. To simplify notation, we make the convention that $P_{0}=I_d$, hence, $\langle xx^*,P_{0}\rangle = \trace(xx^*)=\|x\|^2$, consider the problem
\begin{equation}\label{eq:optimization}
 \text{find }A\in\R^{d\times d}_{\succeq 0}, \quad \text{such that}\quad  \trace(AP_{j}) = \hat{x}(P_j), \; j=0,\ldots,n,
\end{equation}
where $\{P_j\}_{j=1}^n\subset\mathcal{G}_{k,d}$ and $\hat{x}(I_d):= \|x\|^2$. 
Note that \eqref{eq:optimization} is the feasibility problem of a semidefinite programm. For $k=1$, the following result is essentially due to \cite{Gross:2013fk}. The extension to $k\geq 1$ has been derived in \cite{Ehler:2014qc}:
\begin{theorem}[\cite{Ehler:2014qc}]\label{theorem:general ehler}
Suppose that $\|x\|^2$ is known and that $\{(\mathcal{P}_j,\omega_j)\}_{j=1}^n$ is a cubature with nonnegative weights for $\Hom_t(\G_{k,d})$, $t\geq 3$. Let $\mu=\sum_{j=1}^n \omega_j \delta_{\mathcal{P}_j}$ denote the corresponding discrete probability measure, where $\delta_{\mathcal{P}_j}$ is the point measure in $\mathcal{P}_j$. If $\{P_j\}_{j=1}^n\subset\mathcal{G}_{k,d}$ are independently sampled from $\mu$, then with probability at least $1-e^{-\gamma}$, the rank-one matrix $xx^*$ is the unique solution to \eqref{eq:optimization} provided that
\begin{equation}\label{eq:number}
n\geq c_1\gamma  t d^{1+2/t}\log^2(d),
\end{equation}
where $\gamma\geq 1$ is an arbitrary parameter and $c_1$ is a constant, which does not depend on $d$.
\end{theorem}
Hence, choosing random projectors distributed according to discrete probability measures allows us to reconstruct $xx^\top$ with less than $d^2$ many measurements.

\section{Cubatures of varying ranks}\label{sec:6}
In the previous sections, we were dealing with cubatures for Grassmannians of fixed rank. In order to allow for more flexibility, we now aim to remove this restriction, i.e., we shall investigate cubatures for unions of Grassmannians. 
Our present aim is to provide some elementary proofs of some of the results in \cite{Ehler:2014zl} that were derived by the use of representation theoretic concepts.

Given a non-empty set $\mathcal{K}\subset\{1,\ldots,d-1\}$, we define the
corresponding union of Grassmannians by
\begin{equation*}
  \mathcal{G}_{\mathcal{K},d} :=\bigcup_{k\in\mathcal{K}} \mathcal{G}_{k,d} = \{ P \in \R^{d\times d}_{\sym} \;:\; P^{2}=P ,\; \trace(P) \in \mathcal{K}\}.
\end{equation*}
% It is convenient to order the elements of $\K = \{ k_{1},\dots, k_{r} \}$ such
% that
% \begin{equation}
%   \label{eq:convK}
%   \min\{k_{1},d-k_{1}\} \ge \dots \ge \min\{k_{r},d-k_{r}\}
% \end{equation}
% and we will always stick to this convention.
As for a single Grassmannian, the polynomials on $\G_{\K,d}$ are given by
multivariate polynomials in the matrix entries of a given projector
$P \in \G_{\mathcal{K},d}$, i.e., 
% Thus, we define the space of polynomials with degree at most $t$ by
\begin{align}\label{eq:def of pol in union}
  \Pol_{t}(\G_{\mathcal{K},d}) := \{ f|_{\G_{\mathcal{K},d}} : f \in \Pol_t(\R^{d\times d}_{\sym}) \}.
%  \Pol_{t}(\G_{\K,d}):=  \{ &  f:\G_{\K,d} \to \C \;:\; f(P)=g(P_{1,1}, \dots, P_{d,d}) ;\; P=(P_{i,j})_{i,j=1}^{d} \in %\G_{\K,d}; \\
%  & \text{$g$ is polynomial of degree at most $t$} \}.
\end{align}
The dimension of $\Pol_{t}(\G_{\mathcal{K},d})$ is an indicator of the number of points needed to obtain a cubature on $\G_{\mathcal{K},d}$, cf.~\cite{Harpe:2005fk}. To compute this dimension, we shall first derive a lower bound on $\dim(\Pol_{t}(\G_{\K,d}))$, which is relatively straight-forward:
\begin{proposition}\label{the:dimHomt}
  Let $\mathcal K = \{ k_{i}\}_{i=1}^r \subset \{1,\dots,d-1\}$ and $t \in \N_{0}$ be given such that 
  \begin{equation}
    \label{eq:order_K}
    \min\{k_{1},d-k_{1}\} \ge \dots \ge \min\{k_{r},d-k_{r}\}.
  \end{equation}
  Then it holds
  % \begin{equation}\label{eq:dim poly bound}
  %   \sum_{i=1}^{s} \dim(\Pol_{t-s+i}(\mathcal G_{l_{r-s+i},d})) \le
  %   \dim(\Hom_{t}(\mathcal G_{\mathcal K,d})) \le %\dim(\Hom_{t}(\mathscr{H}_{d})) =  
  %   \binom{\frac{d(d+1)}{2}+t-1}{t} .
  % \end{equation}
  \begin{equation}
    \label{eq:dim poly bound}
    \dim(\Pol_{t}(\G_{\K,d})) \geq  \sum_{i=1}^{s}
    \dim(\Pol_{t-i+1}(\G_{k_{i},d})),\quad\text{$s:=\min\{t+1,|\mathcal{K}|\}$.}
  \end{equation}
\end{proposition}
Note that the dimension of each $\Pol_{t-i+1}(\G_{k_{i},d})$ is known, i.e., 
  \begin{equation}\label{eq:dimPolGk}
    \dim(\Pol_{t}(\G_{k,d})) = \sum_{\substack{|\pi| \le t,\\
       \ell(\pi) \le \min\{k,d-k\}}} \mathcal{D}(d,2\pi),
  \end{equation}
  where 
  \begin{equation}\label{eq:dimRep}
\mathcal{D}(d,\pi) = \prod\limits_{1\le i <j \le  \frac{d}{2} }
    \frac{(l_{i}+l_{j})(l_{i}-l_{j})}{(j-i)(d-i-j)} \cdot
    \begin{cases}
      \prod\limits_{1\leq i\leq \frac{d}{2}}
      \frac{2 l_{i}}{d-2 i}
      ,& \text{$d$ odd},\\
      2 
      % \prod\limits_{1\le i <j \le n}
      % \frac{(\l_{i}+\l_{j})(\l_{i}-\l_{j})}{(j-i)(d-i-j)},
      ,& \text{$d$ even and $\pi_{\lfloor \frac d 2 \rfloor}>0$},\\
      1
      % \prod\limits_{1\le i <j \le n}
      % \frac{(\l_{i}+\l_{j})(\l_{i}-\l_{j})}{(j-i)(d-i-j)},
      ,& \text{$d$ even and $\pi_{\lfloor \frac d 2 \rfloor}=0$},
    \end{cases}
  \end{equation}
  with $l_{i}:=\frac{d}{2}+\pi_{i}-i$, for $1\leq i\leq  \frac{d}{2} $, cf.~\cite[Formulas
  (24.29) and (24.41)]{Fulton:1991fk} and \cite{Bachoc:2004fk,Bachoc:2002aa}. Thus, \eqref{eq:dim poly bound} is an explicit lower bound on the dimension of $\Pol_{t}(\G_{\K,d})$.

\begin{proof}[of Proposition \ref{the:dimHomt}]
  We will show that the lower bound \eqref{eq:dim poly bound} is valid for any
  ordering of the indices $k_{1},\dots,k_{r}$. In particular it holds for the
  ordering specified in \eqref{eq:order_K}, which maximizes the right hand side
  over all such lower bounds.

  For $t=0$ or $r=1$, the sum in \eqref{eq:dim poly bound} reduces to a single
  term, so that the lower bound indeed holds. For fixed $t\ge 1$, we verify the
  general case by induction over $r$, where we proceed from $r-1$ to $r$ with
  $r\geq 2$.

  Choose $\{f_i\}_{i=1}^m\subset \Pol_{t}(\R^{d\times d}_{\sym})$ and
  $\{g_j\}_{j=1}^n\subset \Pol_{t-1}(\R^{d\times d}_{\sym})$ such that
  $\{f_i|_{\mathcal G_{k_{1},d}}\}_{i=1}^m$ and
  $\{g_j|_{\mathcal G_{\mathcal{K}\setminus \{k_{1}\},d}}\}_{j=1}^n$ are bases
  for the spaces $\Pol_{t}(\mathcal G_{k_{1},d})$ and
  $\Pol_{t-1}(\mathcal G_{\mathcal K\setminus \{k_{1}\},d})$, respectively. We
  infer that any linear combination
\begin{equation*}
h := \sum_{i=1}^{m} \alpha_{i} f_i|_{\mathcal G_{\mathcal
    K,d}} +\sum_{j=1}^{n} \beta_{j} \big(\Tr(\cdot)-k_{1}\big) g_j|_{\mathcal G_{\mathcal
    K,d}}
\end{equation*}
is contained in $\Pol_{t}(\mathcal G_{\mathcal K,d})$. Suppose now that $h$
vanishes on $\mathcal G_{\mathcal K,d}$. In particular, $h$ vanishes on
$\mathcal G_{k_{1},d}$, so that $\alpha_{i}=0$, $i=1,\dots,m$. Vanishing on $\mathcal G_{\mathcal K \setminus \{k_{1}\},d}$ implies 
$\beta_{j}=0$, $j=1,\dots,n$. Hence, the function system
\[
  \big\{{f_i}|_{\mathcal{G}_{\mathcal{K},d}}\big\}_{i=1}^{m} \cup \big\{
  \big(\Tr(\cdot)-k_{1}\big){g_j}|_{\mathcal{G}_{\mathcal{K},d}} \big\}_{j=1}^{n}
\]
is linearly independent in $\Pol_{t}(\mathcal G_{\mathcal K,d})$. By using 
$s-1 = \min\{t,r-1\}$, we infer by the induction hypothesis
\begin{align*}
  \dim(\Pol_{t}(\mathcal G_{\mathcal K,d})) & \ge \dim(\Pol_{t}(\mathcal G_{k_{1},d})) + \dim(\Pol_{t-1}(\mathcal G_{\mathcal K\setminus \{k_{1}\},d})) \\
  & \ge \dim(\Pol_{t}(\mathcal G_{k_{1},d})) +
  \sum_{i=1}^{(s-1)} \dim(\Pol_{(t-1)-i+1}(\mathcal G_{k_{i+1},d}))  \\
  & = \dim(\Pol_{t}(\mathcal G_{k_{1},d})) +
  \sum_{i=2}^{s} \dim(\Pol_{t-i+1}(\mathcal G_{k_{i},d}))  \\
  & = \sum_{i=1}^{s} \dim(\Pol_{t-i+1}(\mathcal G_{{k_i},d})),
\end{align*}
which proves the lower bound \eqref{eq:dim poly bound}.
\end{proof}
In case $\K=\{k,d-k\}$, we can verify that the lower bound is matched by elementary methods:
\begin{proposition}\label{coro:endlich}
Let $1\leq k\leq d-1$ with $ k\neq  \frac{d}{2}$ and $t\geq 1$. Then it holds
 \begin{equation}\label{eq:imply the last statement}
\Pol_{t}(\G_{k,d}\cup \G_{d-k,d})  \cong \Pol_{t}(\G_{k,d})\oplus \Pol_{t-1}(\G_{d-k,d}).
  \end{equation}
\end{proposition}
%Note that $\dim(\Pol_{t}(\G_{d-k,d}))=\dim(\Pol_{t}(\G_{k,d}))$, for $t\in\N_0$, so that the ordering for the lower bound in \eqref{eq:dim poly bound} with respect to $\mathcal{K}=\{k,d-k\}$ does not matter. 
\begin{proof}
%For $t=1$, we observe that 
%\begin{equation*}
%\dim(\Pol_{1}(\G_{k,d})) = \dim(\Hom_{1}(\R^{d\times d}_{\sym})),
%\end{equation*}
%which then also equals $\dim(\Hom_{1}(\G_{k,d}\cup \G_{d-k,d}))$, so that Theorem \ref{th:structure} yields \eqref{eq:imply the last statement}. 
%
%Suppose now $t\geq 2$. 
We consider the restriction mapping
%Via the rank theorem for the restriction map
\begin{equation*}
|_{\G_{k,d}} \;:\; \Pol_{t}(\G_{k,d}\cup\G_{d-k,d}) \longrightarrow \Pol_t(\G_{k,d}),\quad f\mapsto f|_{\G_{k,d}}
\end{equation*}
and shall verify that the dimension of its nullspace satisfies 
\begin{equation}\label{eq:tbp}
\Null(|_{\G_{k,d}}) = (\trace(\cdot)-k) \Pol_{t-1}(\G_{k,d}\cup\G_{d-k,d}).
\end{equation}
Since $|_{\G_{k,d}} $ is onto and $(\trace(\cdot)-k) \Pol_{t-1}(\G_{k,d}\cup\G_{d-k,d})$ is equivalent to $\Pol_{t-1}(\G_{d-k,d})$, this would imply \eqref{eq:imply the last statement}.
 
It is obvious that the right-hand-side in \eqref{eq:tbp} is contained in $\Null(|_{\G_{k,d}})$. The latter can also be deduced from the lower bounds \eqref{eq:dim poly bound}. For the reverse set inclusion, let $f\in \Null(|_{\G_{k,d}})$. We must now check that $f|_{\G_{d-k,d}}\in \Pol_{t-1}(\G_{d-k,d})$. 

To proceed let us denote $n:=\dim(\Pol_{t}(\G_{k,d}\cup \G_{d-k,d}))$. According to \cite{Ehler:2014zl}, see also \cite{Bachoc:2010aa}, there are $\{X_j\}_{j=1}^n\subset \G_{k,d}\cup \G_{d-k,d}$ and $\{c_j\}_{j=1}^n\subset \R$ such that 
\begin{equation*}
f(P)= \sum_{j=1}^n  c_j \trace(X_jP)^t|_{\G_{k,d}\cup \G_{d-k,d}},\quad P\in\G_{k,d}\cup \G_{d-k,d}.
\end{equation*}
By applying the binomial formula, we observe that 
\begin{equation}\label{eq:the same formula reduction}
f+(-1)^{t+1} f(I_d-\cdot)\in \Pol_{t-1}(\G_{k,d}\cup \G_{d-k,d}).
\end{equation}
Therefore, the assumption $f|_{\G_{k,d}}\equiv 0$ implies that $f(I-\cdot)|_{\mathcal{G}_{k,d}}\in \Pol_{t-1}(\mathcal{G}_{k,d})$. Since $f\mapsto f(I-\cdot)$ is an isomorphism between $\Pol_{t-1}(\G_{k,d})$ and $\Pol_{t-1}(\G_{d-k,d})$, we derive $f|_{\mathcal{G}_{d-k,d}} \in \Pol_{t-1}(\G_{d-k,d})$. Thus, we have verified \eqref{eq:tbp}, which concludes the proof. 
%
%Thus, we have verified that
%\begin{equation*}
%\dim(\Null(|_{\G_{k,d}})) \leq \dim(\Pol_{t-1}(\G_{d-k,d})).
%\end{equation*}
%Due to the rank theorem, we obtain
%\begin{equation*}
%\dim(\Pol_{t}(\G_{k,d}\cup \G_{d-k,d}))\leq  \dim(\Pol_{t}(\G_{k,d}))+ \dim(\Pol_{t-1}(\G_{d-k,d})).
%\end{equation*}
%The matching lower bound in Theorem \ref{the:dimHomt} then completes the proof. 
\end{proof}
Proposition \ref{coro:endlich} shows that, for $\mathcal{K}=\{k,d-k\}$, the inequality in Proposition \ref{the:dimHomt} is an equality. It has been proved in \cite{Ehler:2014zl} that there also holds equality in the general situation:
\begin{theorem}[\cite{Ehler:2014zl}]\label{th:eh}
Let $\mathcal{K} =\{k_{i}\}_{i=1}^r\subset\{1,\ldots,d-1\}$ and $t\in\N_0$ be given such that 
\begin{equation*}
\min\{k_{1},d-k_{1}\} \ge \dots \ge \min\{k_{r},d-k_{r}\}.
\end{equation*}
Then it holds 
\begin{equation}
\Pol_{t}(\G_{\mathcal{K},d}) \cong \bigoplus_{i=1}^{s}\Pol_{t-i+1}(\G_{k_{i},d}),\quad s:=\min\{t+1,|\mathcal{K}|\},
  \end{equation}
\end{theorem}
Compared to our elementary proofs of Propositions \ref{the:dimHomt} and \ref{coro:endlich}, the proof of Theorem \ref{th:eh} presented in \cite{Ehler:2014zl} is much more involved. It makes use of representation theoretic concepts in combination with orthogonally invariant reproducing kernels. 

Note that Theorem \ref{th:eh} reveals that each $f\in \Pol_{t}(\G_{\mathcal{K},d})$ vanishing on $\Pol_{t}(\G_{k_1,d})$ must contain a factor $(\Tr(\cdot) - k_1)\big|_{\G_{\mathcal{K},d}}$, i.e., the restriction mapping $|_{\mathcal{G}_{k_1,d}}$ from $\Pol_{t}(\G_{\mathcal{K},d})$ to $\Pol_{t}(\G_{k_1,d})$, for $t\geq 1$, satisfies
\begin{equation*}
\Null(|_{\mathcal{G}_{k_1,d}}) = (\Tr(\cdot) - k_1) \Pol_{t-1}(\G_{\mathcal{K},d}).
\end{equation*}

Understanding the structure of polynomials on $\G_{\mathcal{K},d}$ is one of the key ingredients to apply the concept of cubatures in the areas of the previous sections. While we now better understand the space $\Pol_{t}(\G_{\mathcal{K},d})$, there is still work to do in order to approximate integrals and functions defined on unions of Grassmannians $\G_{\mathcal{K},d}$, to deal with phase retrieval problems when magnitude is measured in subspaces of varying dimensions, and to derive high dimensional moment reconstructions from marginal moments of varying low dimensions. This shall be addressed in future work.

\begin{acknowledgement}
Thomas Peter was funded by the German Academic Exchange Service (DAAD) through P.R.I.M.E.~57338904. All authors have been supported by the Vienna Science and Technology Fund (WWTF) through project VRG12-009. 
\end{acknowledgement}

\bibliographystyle{amsplain}
\bibliography{../biblio_ehler2}
\end{document}